\documentclass[12pt, a4paper]{amsart}
\usepackage{amsmath}
\usepackage{geometry,amsthm,graphics,tabularx,amssymb,shapepar}
\usepackage{amscd}
\usepackage[all,2cell,dvips]{xy}

\newcommand{\Ad}{{\mathrm{Ad}}}

\newcommand{\GL}{{\mathrm{GL}}}

\newcommand{\Hom}{{\mathrm{Hom}}}

\newcommand{\rk}{{\mathrm{k}}}

\newcommand{\Sp}{{\mathrm{Sp}}}

\newcommand{\tr}{{\mathrm{tr}}}

\newcommand{\vsp}{{\vspace{0.2in}}}

\newcommand{\oJ}{\operatorname{J}}
\newcommand{\oH}{\operatorname{H}}
\newcommand{\oO}{\operatorname{O}}

\newcommand{\oU}{\operatorname{U}}

\renewcommand{\rk}{\mathrm k}

\newcommand{\C}{\mathbb{C}}

\newcommand{\G}{\mathbf{G}}

\newcommand{\E}{\mathbf{E}}

\newcommand{\ve}{{\vee}}

\newcommand{\la}{\langle}
\newcommand{\ra}{\rangle}

\newcommand{\be}{\begin {equation}}
\newcommand{\ee}{\end {equation}}
\newcommand{\bee}{\begin {equation*}}
\newcommand{\eee}{\end {equation*}}

\theoremstyle{Theorem}

\theoremstyle{Theorem}

\newtheorem{thm}{Theorem}[section]

\theoremstyle{Theorem}
\newtheorem{lem}{Lemma}[section]

\newtheorem{thml}[lem]{Theorem}

\newtheorem{prpl}[lem]{Proposition}

\theoremstyle{Theorem}
\newtheorem{prp}{Proposition}[section]

\newtheorem{lemp}[prp]{Lemma}

\theoremstyle{Plain}

\theoremstyle{Definition}

\newtheorem*{rmk}{Remarks}

\begin{document}

\title[Contragredients of representations]{Dual pairs and contragredients of irreducible representations}

\author{Binyong Sun}

\address{Academy of Mathematics and Systems Science, Chinese Academy of
Sciences, Beijing, 100190, China $\&$  Department of Mathematics, Hong Kong University of Science and Technology, Clear Water Bay, Hong Kong}

\email{sun@math.ac.cn}


\thanks{Supported by NSFC (Grants No. 10801126 and 10931006)}

\begin{abstract}
Let $G$ be a classical group $\GL(n)$, $\oU(n)$, $\oO(n)$ or
$\Sp(2n)$, over a non-archimedean local field of characteristic
zero.  It is well known that the contragredient of an irreducible admissible smooth representation
of $G$ is isomorphic to a twist of it by an automorphism of $G$. We prove that
similar results hold for double covers of $G$ which occur in the study of
local theta correspondences.
\end{abstract}

 \maketitle


\section{Introduction and the results}\label{intro}

Fix a non-archimedean local field $\rk$ of characteristic zero. We introduce the following notations in order to treat the four series of classical groups $\GL(n)$, $\oU(n)$, $\oO(n)$ and $\Sp(2n)$ simultaneously. Let $A$ be a $\rk$-algebra and $\tau$ be a $\rk$-algebra involution of
$A$ so that
\[
 (A,\tau)=\left\{
                  \begin{array}{l}
                 (\rk\times \rk, \textrm{the nontrivial
                 automophism}),\\
                  (\textrm{a quadratic field extension of $\rk$}, \textrm{the nontrivial
                 automophism}),\, \textrm{or}\\
                    (\rk, \textrm{the trivial
                 automophism}).
                 \end{array}
            \right.
\]
Let $\epsilon=\pm 1$ and let $E$ be an $\epsilon$-hermitian
$A$-module, namely it is a free $A$-module of finite rank, equipped
with a non-degenerate $\rk$-bilinear map
\[
  \la\,,\,\ra_E:E\times E\rightarrow A
\]
satisfying
\[
     \la u,v\ra_E=\epsilon\la v,u\ra_E^\tau, \quad \la au,v\ra_E=a\la u,
     v\ra_E,\quad a\in A,\, u,v\in E.
\]
Denote by $\oU(E)$ the group of all $A$-module automorphisms of $E$
which preserve the form $\la\,,\,\ra_E$. Depending on $A$ and
$\epsilon$, it is a general linear group, unitary group, orthogonal
group or symplectic group.

Following Moeglin-Vigneras-Waldspurger (\cite[Proposition
4.I.2]{MVW87}), we extend $\oU(E)$ to a larger group, which is
denoted by $\breve{\oU}(E)$, and consisting of pairs
$(g,\delta)\in\GL_\rk(E)\times \{\pm 1\}$ such that either
\[
  \delta=1 \quad\textrm{and}\quad g\in \oU(E),
\]
or
\[
\label{dutilde}
  \left\{
   \begin{array}{ll}
     \delta=-1,&\medskip\\
     g(au)=a^\tau g(u),\quad & a\in A,\, u\in E,\quad \textrm{ and}\medskip\\
     \la gu,gv\ra_E=\la v,u\ra_E,\quad & u,v\in E.
   \end{array}
   \right.
\]
Clearly $\breve{\oU}(E)$ contains $\oU(E)$ as a subgroup of index
two.

In general, if $\pi$ is a representation of a group $H$, and $g$ is
an element of a group which acts on $H$ as automorphisms, we define the twist
$\pi^{{g}}$ to be the representation of $H$ which has the same
underlying space as that of $\pi$, and whose action is given by
\[
   \pi^{{g}}(h):=\pi({g}.h), \quad h\in H.
\]
If $\breve H$ is a group containing $H$ as a subgroup of index two,
we always let it act on $H$ by conjugations: \[
  \Ad: \breve H\times H\rightarrow H,\quad (\breve g, x)\mapsto
  \Ad_{\breve g}(x):=\breve{g}x\breve{g}^{-1}.
\]

It is a classical result in linear algebra that for any $\breve g\in \breve{\oU}(E)\setminus \oU(E)$ and any $x\in \oU(E)$, $\breve g x \breve{g}^{-1}$ is conjugate to $x^{-1}$ inside $\oU(E)$. Then by the localization principle of Bernstein and Zelevinsky (\cite[Theorem 6.9 and Theorem 6.15 A]{BZ76}), for every conjugation invariant generalized function $f$ on $\oU(E)$, and every $\breve g\in \breve{\oU}(E)\setminus \oU(E)$, we have that
\begin{equation}\label{finve}
  f(\breve g x \breve{g}^{-1})=f(x^{-1})
\end{equation}
as generalized functions on $\oU(E)$. For the usual notion of
generalized functions, see \cite[Section 2]{Sun08}, for example.
We get the following well know result by \eqref{finve} and by
considering characters of irreducible admissible smooth
representations (which are conjugation invariant generalized functions).

\begin{thm}\label{contra}
$($\cite[Theorem 4.II.1]{MVW87}$)$ Let $\breve{g}\in \breve{\oU}(E)\setminus \oU(E)$, and let $\pi$ be
an irreducible admissible smooth representation of $\oU(E)$. Then
$\pi^{\ve}$ is isomorphic to $\pi^{\breve{g}}$.
\end{thm}

Here and as usual, we use ``$\,^\ve$" to indicate the contragredient of an admissible
smooth representation of a totally disconnected locally compact group.

\vsp

If $E$ is a symplectic space, i.e., if $\epsilon=-1$ and $A=\rk$,
then $\breve{\Sp}(E):=\breve{\oU}(E)$ equals to the subgroup of
$\operatorname{GSp}(E)$ with similitudes $\pm1$. Denote by
\begin{equation}\label{meta}
   1\rightarrow \{\pm 1\}\rightarrow
   \widetilde{\Sp}(E)\rightarrow
   \Sp(E)\rightarrow 1
\end{equation}
the metaplectic cover of the symplectic group $\Sp(E)$. It is
shown in \cite[Page 36]{MVW87} that there is a unique continuous
action
\begin{equation}\label{liftact0}
   \widetilde{\Ad}:  \breve{\Sp}(E)\times \widetilde{\Sp}(E)\rightarrow \widetilde{\Sp}(E)
\end{equation}
of $\breve{\Sp}(E)$ on $\widetilde{\Sp}(E)$ as group automorphisms which lifts the adjoint action
\[
  \Ad: \breve{\Sp}(E)\times\Sp(E)\rightarrow \Sp(E)
\]
and leaves the central element $-1\in \widetilde{\Sp}(E)$ fixed.

We first extend Theorem \ref{contra} to the case of metaplectic groups:

\begin{thm}\label{contras}
Assume that $E$ is a symplectic space. Let $\breve{g}\in \breve{\Sp}(E)\setminus
\Sp(E)$, and let $\pi$ be a genuine irreducible admissible smooth
representation of $\widetilde{\Sp}(E)$. Then $\pi^{\ve}$ is
isomorphic to $\pi^{\breve{g}}$.
\end{thm}

Here and henceforth, ``genuine" means that the central element $-1\in \widetilde{\Sp}(E)$
acts via the scalar multiplication by $-1$.

\begin{rmk}
(a) Under the hypothesis that the character of $\pi$ is a locally integrable function, Theorem \ref{contras} is proved in \cite[Theorem 4.II.2]{MVW87}.

(b) Harish-Chandra proves locally integrability of irreducible characters for p-adic linear reductive groups. But metaplectic groups are not included in his setting.

(c) The proofs of both Theorem \ref{contra} and Theorem \ref{contras} do not depend on locally integrability of irreducible characters.
\end{rmk}

\vsp

Now we consider dual pairs. Write $\epsilon':=-\epsilon$, and
let $(E',\la\,,\,\ra_{E'})$ be an $\epsilon'$-hermitian $A$-module.
Then
\[
   \E:=E\otimes_A E'
\]
is a skew-hermitian $A$-module under the form
\[
  \la u\otimes u', v\otimes v'\ra_\E:=\la u, v\ra_E \, \la
  u',v'\ra_{E'}.
\]
Write $\E_\rk:=\E$, viewed as a $\rk$-symplectic space under the form
\[
  \la u,v\ra_{\E_\rk}:=\tr_{A/\rk}(\la u,v\ra_\E).
\]
Put
\[
   G:=\oU(E),\quad
   \breve{G}:=\breve{\oU}(E),\quad G':=\oU(E'),\quad \breve{G'}:=\breve{\oU}(E').
\]
The group $G$ obviously maps to the symplectic group $\Sp(\E_\rk)$. Denote the
fiber product
\[
  \widetilde G:=\widetilde{\Sp}(\E_\rk)\times_{\Sp(\E_\rk)} G,
\]
which is a double cover of $G$.

In what follows, we define an action
\begin{equation}\label{liftact}
   \widetilde{\Ad}:  \breve{G}\times \widetilde{G}\rightarrow
   \widetilde{G}
\end{equation}
which lifts the adjoint action
\[
  \Ad: \breve{G}\times G\rightarrow G
\]
and fixes the central element $-1\in \widetilde{G}$. Let
$\breve{g}=(g,\delta)\in \breve{G}$. Choose an arbitrary element
$(g',\delta)\in \breve{G'}$. Then
\[
  \breve{\mathbf g}:=(g\otimes g',\delta)\in \breve{\Sp}(\E_\rk),
\]
and the automorphism
\begin{equation}\label{actprod}
  \widetilde{\Ad}_{\breve{\mathbf g}}\times \Ad_{\breve g}: \widetilde{\Sp}(\E_\rk)\times
  G\rightarrow \widetilde{\Sp}(\E_\rk)\times G
\end{equation}
leaves the subgroup $\widetilde{G}$ stable. It restricts to an
automorphism
\begin{equation}\label{actrestrict}
  \widetilde{\Ad}_{\breve g}: \widetilde{G}\rightarrow
  \widetilde{G}
\end{equation}
which is independent of the choice of $g'$. We obtain
(\ref{liftact}) by gluing (\ref{actrestrict}) for all $\breve{g}\in
\breve{G}$.

We unify and generalize Theorem \ref{contra} and Theorem \ref{contras} as follows:
\begin{thm}\label{contra2}
Let $\breve{g}\in \breve G\setminus G$, and let $\pi$ be a genuine
irreducible admissible smooth representation of $\widetilde G$. Then
$\pi^{\ve}$ is isomorphic to $\pi^{\breve{g}}$.
\end{thm}
Theorem \ref{contra2} is specified to Theorem \ref{contra} when
$E'=0$, and is specified to Theorem \ref{contras} when $E'=A=\rk$ and
$\epsilon'=1$.

Theorem \ref{contra2} has the following consequence, which is
known to experts (up to a proof of Theorem \ref{contras}). But as far as the author knows, no proof of it in full generality was written down in the literature. 
\begin{thm}\label{contra4}
Denote by $\omega_\psi$ the smooth oscillator representation of
$\widetilde{\Sp}(\E_\rk)$ corresponding to a non-trivial character
$\psi$ of $\rk$. Then for all genuine irreducible admissible
smooth representation $\pi$ of $\widetilde{G}$, and $\pi'$ of
$\widetilde{G}'$, the equality
\[
   \dim \Hom_{G\times G'}(\omega_\psi\otimes \pi\otimes \pi',\C)=\dim \Hom_{G\times G'}(\omega_\psi^\ve \otimes \pi^\ve\otimes
   \pi'^\ve,\C)
\]
holds.
\end{thm}

Here $\widetilde G':=\widetilde{\Sp}(\E_\rk)\times_{\Sp(\E_\rk)} G'$
is a double cover of $G'$.  Note that both $\omega_\psi\otimes
\pi\otimes \pi'$ and $\omega_\psi^\ve \otimes \pi^\ve\otimes
   \pi'^\ve$, which are originally representations of $\widetilde{G}\times \widetilde{G}'$, descend to representations of $G\times G'$.

\begin{rmk}
(a) In a sequential paper (\cite{LST}), Theorem \ref{contra4} will
be used to prove multiplicity preservations in theta
correspondences (for all residue characteristics), i.e., the
dimension in Theorem \ref{contra4} is at most one. This is the main reason for the author to provide a detailed proof of Theorem \ref{contra4} in this note. 

(b)The archimedean analogs of Theorem \ref{contra2} and Theorem
\ref{contra4} are proved by T. Przebinda in \cite{Pr88}. His method is different from ours. (He uses the Langlands classification.) 

(c) As showed in \cite{Pr88}, assuming Howe duality conjecture,
Theorem \ref{contra4} implies that theta lifting maps hermitian
representations to hermitian representations.
\end{rmk}

\section{A generalization of Theorem \ref{contras}}

\subsection{Skew hermitian modules and Jacobi groups}\label{sub1}  Assume that $\epsilon=-1$. As in the last section,
$E$ is an $\epsilon$-hermitian $A$-module, and $E_\rk:=E$ is a
symplectic space under the form
\[
  \la u,v\ra_{E_\rk}:=\tr_{A/\rk}(\la u,v\ra_E).
\]
Denote by
\[
  \oH(E):=E_\rk\times \rk
\]
the Heisenberg group associated to $E_\rk$, whose multiplication is
given by
\[
   (u,t)(u',t'):=(u+u', t+t'+\la u,u'\ra_{E_\rk}).
\]
The group $\breve{\oU}(E)$ acts on $\oH(E)$ as group automorphisms
by
\begin{equation}\label{actbr}
  (g, \delta).(u,t):=(gu,\delta t).
\end{equation}
It defines a semidirect product
\[
  \breve{\oJ}(E):=\breve{\oU}(E)\ltimes \oH(E),
\]
which contains
\[
    \oJ(E):={\oU}(E)\ltimes \oH(E)
\]
as a subgroup of index two.

The results of this note depend heavily on following

\begin{lem}\label{sun}(\cite[Theorem D]{Sun08})
Let $f$ be a generalized function on $\oJ(E)$. If it is invariant
under conjugations by $\oU(E)$, i.e.,
\[
    f(gx g^{-1})=f(x),\quad \textrm{for all } g\in\oU(E),
\]
then
\[
    f(\breve{g}x\breve{g}^{-1})=f(x^{-1}),\quad \textrm{for all } \breve{g}\in
    \breve{\oU}(E)\setminus \oU(E).
\]
\end{lem}

Actually, we only need the following lemma, which is much weaker
than Lemma \ref{sun}.

\begin{lem}\label{sun2}
Let $f$ be a conjugation invariant generalized function on
$\oJ(E)$. Then
\[
    f(\breve{g}x\breve{g}^{-1})=f(x^{-1}),\quad \textrm{for all } \breve{g}\in
    \breve{\oJ}(E)\setminus \oJ(E).
\]
\end{lem}

Lemma \ref{sun2} has the following consequence.

\begin{prpl}\label{contraj}
Let $\breve{g}\in \breve{\oJ}(E)\setminus \oJ(E)$, and let $\pi$
be an irreducible admissible smooth representation of $\oJ(E)$.
Then $\pi^{\ve}$ is isomorphic to $\pi^{\breve{g}}$.
\end{prpl}
\begin{proof}
Denote by $f$ the character of $\pi$, which is thus a conjugation
invariant generalized function on $\oJ(E)$. Therefore
\begin{equation}\label{feq}
  f(\breve{g}x\breve{g}^{-1})=f(x^{-1})
\end{equation}
by Lemma \ref{sun2}. The left hand side of (\ref{feq}) is the
character of $\pi^{\breve g}$, and the right hand side is the
character of $\pi^\ve$. Therefore  $\pi^{\breve g}$ and $\pi^\ve$
have the same character, and they are thus isomorphic to each
other.
\end{proof}

\subsection{Metaplectic groups and a generalization of Theorem \ref{contras}}\label{pcontras}
We continue with the notation of Section \ref{sub1}. Denote by
\[
  \widetilde{\oU}(E):=\widetilde{\Sp}(E_\rk)\times_{\Sp(E_\rk)} \oU(E)
\]
the double cover of $\oU(E)$ induced by the metaplectic cover
\begin{equation}\label{meta}
   1\rightarrow \{\pm 1\}\rightarrow
   \widetilde{\Sp}(E_\rk)\rightarrow
   \Sp(E_\rk)\rightarrow 1.
\end{equation}
As in \eqref{liftact}, we have an action
\begin{equation}\label{liftact2}
   \widetilde{\Ad}:  \breve{\oU}(E)\times \widetilde{\oU}(E)\rightarrow
   \widetilde{\oU}(E).
\end{equation}

Theorem \ref{contras} is one of the three cases (corresponding to
the three cases of $A$) of the following theorem.
\begin{thml}\label{contram}
Assume that $\epsilon=-1$. Let $\breve{g}\in
\breve{\oU}(E)\setminus \oU(E)$, and let $\pi$ be a genuine
irreducible admissible smooth representation of
$\widetilde{\oU}(E)$. Then $\pi^{\ve}$ is isomorphic to
$\pi^{\breve{g}}$.
\end{thml}
\begin{proof} Denote by ${\omega}_\psi$ the smooth oscillator
representation of $\widetilde{\Sp}(E_\rk)\ltimes \oH(E)$
corresponding to a nontrivial character $\psi$ of $\rk$. Up to
isomorphism, this is the only genuine smooth representation which,
as a representation of $\oH(E)$, is irreducible and has central
character $\psi$.

Both $\omega_{\psi}$ and $\pi$ are viewed as smooth
representations of $\widetilde{\oJ}(E):=\widetilde{\oU}(E)\ltimes
\oH(E)$, via the restriction and the inflation, respectively. The
tensor product $\omega_{\psi}\otimes\pi$ descends to an
irreducible admissible smooth representation of $\oJ(E)$
(\cite[Lemma 5.3 ]{Sun08}).

The actions of $\breve{\oU}(E)$ on $\widetilde{\oU}(E)$, $\oU(E)$
and $\oH(E)$ induce its actions on the semidirect products
$\widetilde{\oJ}(E)$ and $\oJ(E)$. By Proposition \ref{contraj}, as irreducible admissible smooth representations of $\oJ(E)$, 
\[
   (\omega_{\psi}\otimes\pi)^{\breve{g}}\cong
   (\omega_{\psi}\otimes\pi)^\ve,
\]
or the same,
\[
   \omega_{\psi}^{\breve g}\otimes\pi^{\breve{g}}\cong
   \omega_{\psi}^\ve\otimes\pi^\ve.
\]
Note that
\[
  \omega_{\psi}^{\breve g}\cong \omega_{\psi}^\ve
\]
as smooth representations of $\widetilde{\oJ}(E)$. (This is a special case of Lemma \ref{contrao} of the next section.)  Therefore 
\begin{equation}\label{iso1}
   \omega_{\psi}^\ve\otimes\pi^{\breve{g}}\cong
   \omega_{\psi}^\ve\otimes\pi^\ve.
\end{equation}

As in the proof of \cite[Lemma 5.3 ]{Sun08}, we have that
\begin{equation}\label{iso2}
   \pi^{\breve{g}}\cong\Hom_{\oH(E)}(\omega_\psi^\ve, \omega_\psi^\ve\otimes \pi^{\breve g}).
\end{equation}
Here the right hand side carries the action of $\widetilde{\oU}(E)$
given by
\[
  (\tilde{g}.\phi)(v):=g.(\phi(\tilde{g}^{-1}.v)),
\]
where
\[
   \tilde{g}\in\widetilde{\oU}(E), \quad\phi\in  \Hom_{\oH}(\omega_\psi^\ve, \omega_\psi^\ve\otimes \pi^{\breve g}), \quad  v\in \omega_\psi^\ve,
\]
and $g$ is the image of $\tilde g$ under the covering map
$\widetilde{\oU}(E)\rightarrow \oU(E)$. Similarly,
\begin{equation}\label{iso3}
   \pi^\ve\cong\Hom_{\oH}(\omega_\psi^\ve, \omega_\psi^\ve\otimes \pi^\ve).
\end{equation}
We finish the proof by combining (\ref{iso1}), (\ref{iso2}) and
(\ref{iso3}).
\end{proof}

\section{Proofs of Theorem \ref{contra2} and Theorem
\ref{contra4}}

\subsection{Proof of Theorem \ref{contra2} for symplecitc groups}
Now we return to use the notation of Section \ref{intro}. First assume that $A=\rk$ and
$\epsilon=-1$. Then $G$ is a symplectic group and is thus perfect,
i.e., $G$ equals to its own commutator group. Consequently, there
is only one action of $\breve G$ on $\widetilde G$ which lifts the
adjoint action and fixes the central element $-1\in \widetilde G$.
There are two cases.

Case 1: The covering map $\widetilde G\rightarrow G$ splits. Then $
\widetilde G=G\times \{\pm 1\}$, and Theorem \ref{contra2} is one
case of Theorem \ref{contra}.

Case 2: The covering map $\widetilde G\rightarrow G$ does not split.
Then $G=\widetilde{\Sp}(E)$ (\cite[Theorem 10.4]{Moo68}), and
Theorem \ref{contra2} is one case of Theorem \ref{contras}.

\subsection{Proof of Theorem \ref{contra2} when $A\neq \rk$} Assume
that $A\neq \rk$. Then $\oU(\E)$ is a general linear group or a
unitary group.
\begin{lemp}\label{contrac}
There exists a genuine character  on $\widetilde{\oU}(\E)$.
\end{lemp}
\begin{proof}
It is well known that the exact sequence
\[
   1\rightarrow \C^\times \rightarrow (\widetilde{\Sp}(\E_\rk)\times
   \C^\times)/\operatorname{diag}(\{\pm 1\})\rightarrow \Sp(\E_\rk)\rightarrow 1
\]
splits continuously over $\oU(\E)$ (this is trivial for general
linear groups, and for unitary groups, see \cite[Proposition 4.1]{Ku} or \cite[Section 1]{HKS96}).
Write $\iota$ for such a splitting and write $p:
\widetilde{\oU}(\E)\rightarrow \oU(\E)$ for the covering map. Then
\[
   x\in \widetilde{\oU}(\E)\mapsto x^{-1}\,\iota(p(x))\in \C^\times
\]
is a genuine character.
\end{proof}

\begin{lemp}\label{contrac2}
There exists a genuine character $\chi$ of $\widetilde G$ such that
$\chi^{\breve g}=\chi^{-1}$ for all $\breve g\in \breve G\setminus
G$.
\end{lemp}
\begin{proof}
As in Section \ref{intro},  let
\[
   \breve{g}=(g,-1)\in \breve{G}\setminus G, \quad (g',-1)\in
\breve{G'}\setminus G',
\]
and write
\[
  \breve{\mathbf g}:=(g\otimes g',-1)\in \breve{\oU}(\E)\setminus \oU(\E).
\]
It is obvious that the diagram
\begin{equation}\label{cd1}
   \begin{CD}
     \widetilde{\oU}(\E)@>\widetilde{\Ad}_{\breve{\mathbf g}}>>     \widetilde{\oU}(\E) \\
      @AAA                 @ AAA\\
     \widetilde{G}@>\widetilde{\Ad}_{\breve{g}}>>     \widetilde{G} \\
   \end{CD}
\end{equation}
commutes.

Take a character $\chi_\E$ as in Lemma \ref{contrac}, and denote by
$\chi$ its restriction to $\widetilde G$. Then
\begin{eqnarray*}
   && \quad\,\chi^{\breve g}\\
   && =(\chi_\E|_{\widetilde G})^{\breve g}\\
   && =(\chi_\E^{\breve{\mathbf g}})|_{\widetilde G}\qquad\,\, \quad\textrm{by commutativity of (\ref{cd1})}\\
   && =(\chi_\E^{-1})|_{\widetilde G} \qquad \quad\textrm{by Theorem \ref{contram}}\\
   && =\chi^{-1}.
\end{eqnarray*}

\end{proof}
Fix $\chi$ as in Lemma \ref{contrac2}. Let $\breve{g}\in \breve
G\setminus G$, and let $\pi$ be a genuine irreducible admissible
smooth representation of $\widetilde G$. Then $\pi\otimes \chi$ descends to 
an irreducible admissible smooth representation of $G$. By Theorem
\ref{contra},
\[
  (\pi\otimes \chi)^{\breve g}\cong (\pi\otimes \chi)^\ve,
\]
or equivalently,
\[
  \pi^{\breve g}\otimes \chi^{\breve g}\cong \pi^\ve\otimes \chi^{-1},
\]
Therefore, $\pi^{\breve g}\cong \pi^\ve$  since $\chi^{\breve
g}=\chi^{-1}$. This proves Theorem \ref{contra2} when $A\neq \rk$.

\subsection{Proof of Theorem \ref{contra2} for orthogonal groups}
Assume that $A=\rk$ and $\epsilon=1$, i.e., $G$ is an orthogonal
group. In what follows, we show that Lemma \ref{contrac2} still
holds in this case. Fix a complete polarization
\[
  E'=E'_+\oplus E'_-
\]
of the symplectic space $E'$. Then
\[
   \E=\E_+\oplus \E_-
\]
is a complete polarization of the symplectic space $\E$, where
$\E_{\pm}:=E\otimes E'_{\pm}$. Depending on this polarization, we
define a skew-hermitian $\rk\times \rk$-module $\E'$ as follows. As
an abelian group, $\E'=\E$. The scalar multiplication is given by
\[
   (ae_1 + be_2)(u+v):=au+bv, \quad a,b\in \rk, \, u\in \E_+, v\in \E_-,
\]
where
\[
  e_1:=(1,0)\quad \textrm{and}\quad e_2:=(0,1)
\]
are the two idempotent elements of $\rk\times \rk$. The
skew-hermitian form is given by
\[
  \la u_+ + u_-,v_+ + v_-\ra_{\E'}:=\la u_+,v_-\ra_{\E}\, e_1 +\la u_-,v_+\ra_{\E} \,e_2,
\]
where $u_+, v_+\in \E_+$, $u_-, v_-\in \E_-$.

Let $\breve g=(g,-1)\in \breve G\setminus G$. Choose an element
$(g',-1)\in \breve G'\setminus G'$ such that
\[
  g'(E'_+)=E'_-\quad \textrm{and}\quad g'(E'_-)=E'_+.
\]
Then
\[
   \breve{\mathbf g}:=(g\otimes g',-1)\in \breve{\oU}(\E')\setminus
   \oU(\E'),
\]
and the diagram
\[
   \begin{CD}
     \widetilde{\oU}(\E')@>\widetilde{\Ad}_{\breve{\mathbf g}}>>     \widetilde{\oU}(\E') \\
      @AAA                 @ AAA\\
     \widetilde{G}@>\widetilde{\Ad}_{\breve{g}}>>     \widetilde{G} \\
   \end{CD}
\]
commutes.

Take a genuine character $\chi_{\E'}$ of $\widetilde{\oU}(\E')$ as
in Lemma \ref{contrac}, and denote by $\chi$ its restriction to
$\widetilde G$. Then as in the proof of Lemma \ref{contrac2}, we
show that $\chi$ fulfills the requirement of Lemma \ref{contrac2}.
Now we argue as in the end of the last subsection, and prove Theorem
\ref{contra2} for orthogonal groups.

The proof of Theorem \ref{contra2} is now complete by combining this
subsection with the last two subsections.

\subsection{Proof of Theorem \ref{contra4}}
The group
\[
   \breve{\G}:=\breve{G}\times_{\{\pm 1\}} \breve{G'}=\{(g,g',\delta)\mid (g,\delta)\in
  \breve{G},\,(g',\delta)\in\breve{G'}\}
\]
contains
\[
  \G:=G\times G'
\]
as a subgroup of index two. Define a homomorphism
\[
   \xi: \breve{\G}\rightarrow \breve{\Sp}(\E_\rk), \quad (g,g',\delta)\mapsto (g\otimes g',\delta).
\]
By using the covering map 
\[
 \widetilde G\times \widetilde
  G'\rightarrow \G=G\times G'
\]
and the map 
\[
 \xi|_{\G}:\G\rightarrow \Sp(\E_\rk),
\]
we form the semidirect product
\[
  (\widetilde G\times \widetilde
  G')\ltimes \oH(\E)
\]
as in Section \ref{sub1}. Let $\breve{\G}$ act on $(\widetilde
G\times \widetilde G')\ltimes \oH(\E)$ as group automorphisms by
\begin{equation}\label{actionbg}
  \breve{\mathbf g}. (x,y,z):=(\widetilde{\Ad}_{\breve g}(x), \,\widetilde{\Ad}_{\breve
  g'}(y), \, \xi(\breve{\mathbf g}).z),
\end{equation}
where
\[
  \breve{\mathbf g}=(g,g',\delta), \quad \breve g=(g,\delta), \quad \breve
  g'=(g',\delta),
\]
and the last term of the right hand side of \eqref{actionbg} is defined as in \eqref{actbr}.

Let $\omega_\psi$, $\pi$ and $\pi'$ be as in Theorem \ref{contra4}.

\begin{lemp}\label{contrao}
View $\omega_\psi$ as an admissible smooth representation of
$(\widetilde G\times \widetilde G')\ltimes \oH(\E)$ (via the restriction). Then for every
$\breve{\mathbf g}\in \breve{\G}\setminus \G$, we have
\[
  \omega_\psi^\ve\cong \omega_\psi^{\breve{\mathbf g}}.
\]

\end{lemp}
\begin{proof}
Recall that the group $\breve{\Sp}(\E_\rk)$ acts on
$\widetilde{\Sp}(\E_\rk)\ltimes \oH(\E)$ diagonally through its
action on the two factors. We have
\begin{equation}\label{congo}
\omega_\psi^\ve\cong \omega_\psi^{\xi(\breve{\mathbf g})}
\end{equation}
as smooth oscillator representations of
$\widetilde{\Sp}(\E_\rk)\ltimes \oH(\E)$, since both corresponding
to the character $\psi^{-1}$. We prove the lemma by restricting both
sides of (\ref{congo}) to the group $(\widetilde G\times \widetilde
G')\ltimes \oH(\E)$.
\end{proof}

\begin{lemp}\label{contrapi}
View $\pi$ and $\pi'$ as admissible smooth representations of
$(\widetilde G\times \widetilde G')\ltimes \oH(\E)$ (via the inflations). Then for every
$\breve{\mathbf g}\in \breve{\G}\setminus \G$, we have
\begin{equation}\label{congpi}
  \pi^\ve\cong \pi^{\breve{\mathbf g}}\quad \textrm{and}\quad \pi'^\ve\cong \pi'^{\breve{\mathbf
  g}}.
\end{equation}

\end{lemp}
\begin{proof}
Write
\[
  \breve{\mathbf g}=(g,g',-1), \quad \textrm{and}\quad\breve g=(g,-1).
\]
By Theorem \ref{contra2},
\[
  \pi^\ve\cong \pi^{\breve{g}}
\]
as irreducible admissible smooth representations of $\widetilde G$.
Pull back this isomorphism to the group $(\widetilde G\times
\widetilde G')\ltimes \oH(\E)$, we obtain the first isomorphism of
(\ref{congpi}). The second isomorphism follows similarly.
\end{proof}

\begin{lemp}\label{contrat}
For every $\breve{\mathbf g}\in \breve{\G}\setminus \G$,
we have
\begin{equation}\label{tensor}
 \omega_\psi^\ve \otimes \pi^\ve\otimes \pi'^\ve\cong (\omega_\psi\otimes \pi\otimes
\pi')^{\breve{\mathbf g}}
\end{equation}
as smooth representations of $(\widetilde G\times \widetilde
G')\ltimes \oH(\E)$.
\end{lemp}
\begin{proof}
This is a combination of Lemma \ref{contrao} and Lemma
\ref{contrapi}.
\end{proof}

Fix an element $\breve{\mathbf g}\in \breve{\G}\setminus \G$. Since the action of $\breve{\mathbf g}$ stabilizes the
subgroup $\widetilde G\times \widetilde G'$ of $(\widetilde G\times
\widetilde G')\ltimes \oH(\E)$, we have
\begin{equation}\label{home}
  \Hom_{\widetilde G\times \widetilde G'}(\omega_\psi\otimes \pi\otimes
  \pi',\C)=\Hom_{\widetilde G\times \widetilde G'}((\omega_\psi\otimes \pi\otimes
  \pi')^{\breve{\mathbf g}},\C).
\end{equation}
Now Theorem \ref{contra4} is a consequence of (\ref{tensor}) and
(\ref{home}).

\end{document}